\documentclass[11pt,a4paper]{article}

\usepackage[british]{babel}
\usepackage[margin=2.5cm]{geometry}
\usepackage{amsmath,amssymb,amsthm,mathtools}
\usepackage{xcolor}
\usepackage{enumitem}
\usepackage{hyperref}
\usepackage{microtype}

\hypersetup{
    colorlinks=true,
    linkcolor=blue!55!black,
    citecolor=blue!55!black,
    urlcolor=blue!55!black
}

\newtheorem{theorem}{Theorem}[section]
\newtheorem{lemma}[theorem]{Lemma}
\newtheorem{proposition}[theorem]{Proposition}
\newtheorem{corollary}[theorem]{Corollary}
\newtheorem{definition}[theorem]{Definition}
\newtheorem{conjecture}[theorem]{Conjecture}
\newtheorem{remark}[theorem]{Remark}

\title{Quasi-kernels in Hereditary Classes and Applications to Break Digraphs}
\author{Jiangdong Ai\thanks{School of Mathematical Sciences and LPMC, Nankai University. {\tt jd@nankai.edu.cn}.}
\hspace{2mm}
Tianyu Huang\thanks{School of Mathematical Sciences and LPMC, Nankai University. {\tt tyhuang@mail.nankai.edu.cn}.}
\hspace{2mm}
Xiangzhou Liu\thanks{School of Mathematical Sciences and LPMC, Nankai University. {\tt  
i19991210@163.com}.}
\hspace{2mm} Chaoliang Tang\thanks{Shanghai Center for Mathematics Science, Fudan University. {\tt cltang22@m.fudan.edu.cn}}
\hspace{2mm} 
Zongqun Xu\thanks{School of Mathematical Sciences and LPMC, Nankai University. {\tt xu\_zongqun@163.com}}
}

\begin{document}
\maketitle

\begin{abstract}
Recently,  Nguyen, Seymour and Scott verified the small quasi-kernel conjecture for split digraphs, and initiated the study of quasi-kernels in break digraphs. Following their research, we introduce a weighted half-neighborhood property for hereditary classes of oriented graphs and show that it gives a \(2n/3\) bound of small quasi-kernel for break digraphs.  We also
record two stronger \(n/2\) results for special classes of break digraphs.
Finally, using the same framework we also prove that every digraph on \(n\) vertices has a quasi-kernel \(Q\) with \(|N_D^+[Q]|\ge \sqrt n\).
\end{abstract}
\section{Introduction}

A \emph{quasi-kernel} of a digraph \(D\) is an independent set
\(Q\subseteq V(D)\) such that every vertex of \(D\) is reachable from some
vertex of \(Q\) by a directed path of length at most \(2\).
Quasi-kernels were introduced by Chv\'atal and Lov\'asz~\cite{CL}, who proved that
every finite digraph admits a quasi-kernel. Since then, quasi-kernels have been
studied extensively in connection with domination-type problems in directed
graphs.

One central direction in the area is to understand how small a quasi-kernel can be. P.L. Erd\H{o}s and Sz\'{e}kely made the following conjecture on the existence of small quasi-kernels. 

\begin{conjecture}[Small Quasi-kernel Conjecture \cite{ES}, 1976]\label{small}
If $D$ is a sink-free digraph, then $D$ has a quasi-kernel $Q$ with $|Q|\leq \frac{1}{2}|V(D)|$.
\end{conjecture}


Conjecture~\ref{small} is wide open: the best bound that works for all sink-free digraphs appears to be $|Q|\le n-\sqrt{n\log n}/4$ \cite{Spiro}. However, there have been substantial results that confirm that Conjecture~\ref{small} holds on certain classes of digraphs. 
Heard and Huang \cite{heard2008disjoint}  showed that a sink-free digraph $D$ has two disjoint quasi-kernels if $D$ is semicomplete multipartite, quasi-transitive, or locally semicomplete. As a consequence, Conjecture ~\ref{small} is true for these three classes of digraphs. Van Hulst \cite{Hulst2021KernelsAS} showed that Conjecture~\ref{small} holds for all digraphs containing kernels. Kostochka, Luo and Shan \cite{MR4477848} proved that Conjecture~\ref{small} holds for digraphs with chromatic number at most~4. In particular, Ai, Gerke, Gutin, Yeo and Zhou~\cite{AGGYZ}
verified the small quasi-kernel problem for one-way split
digraphs. More recently, Seymour, together with Nguyen and Scott~\cite{NSS} verified it for split digraphs, and also initiated
the study of quasi-kernels in \emph{break digraphs}. Recall that an
oriented graph \(D\) is a break digraph if its vertex set admits a partition
\[
V(D)=S\dot\cup T
\]
such that \(D[S]\) is acyclic and \(D[T]\) is a tournament.
Nguyen, Scott and Seymour proved a weighted large quasi-kernel theorem for
break digraphs, showing that for every nonnegative weight function
\(f:V(D)\to\mathbb Z_{\ge0}\), there exists a quasi-kernel \(Q\) satisfying
\[
f(N_D^+[Q])\ge \frac{f(V(D))}{2}.
\]
Their result suggests that weighted large quasi-kernel phenomena may extend far
beyond break digraphs themselves.

Motivated by this perspective, we develop a general hereditary framework for
weighted large quasi-kernels. The main idea is to isolate an abstract weighted
half-neighborhood property and study its structural consequences in hereditary
classes of oriented graphs. Our first main result is the following.

\begin{theorem}
Let \(\mathcal G\) be a hereditary class satisfying the weighted
half-neighborhood property. Then every source-free digraph
\(D\in\mathcal G\) on \(n\) vertices has a quasi-kernel \(Q\) such that
\[
|Q|\le \frac{2n}{3}.
\]
\end{theorem}

The proof combines the weighted theorem of Nguyen--Scott--Seymour with a new
reduction involving sinks and a matching argument. Since the class of break
digraphs is hereditary, their theorem immediately yields the following
consequence.

\begin{corollary}
Every source-free break digraph on \(n\) vertices has a quasi-kernel of size at
most \(2n/3\).
\end{corollary}

We also investigate more structured subclasses of break digraphs. In
particular, we prove that one-way break digraphs admit quasi-kernels of size at
most \(n/2\), and we obtain another \(n/2\) theorem under a small out-kernel
assumption on the acyclic side of a break.

Besides small quasi-kernels, another important direction concerns the existence
of quasi-kernels with large closed out-neighborhoods. This viewpoint is closely
related to the \emph{Large Quasi-kernel Conjecture}, studied by
Spiro~\cite{Spiro}, which predicts that every digraph \(D\) has a quasi-kernel
\(Q\) satisfying
\[
|N_D^+[Q]|\ge \frac{|V(D)|}{2}.
\]
Spiro proved that every digraph contains a quasi-kernel \(Q\) with
\[
|N_D^+[Q]|\ge |V(D)|^{1/3}.
\]

Our second main contribution gives a substantial improvement of this general
bound. We prove a general theorem for hereditary classes, from
which we derive the following consequence.

\begin{theorem}
Every digraph \(D\) on \(n\) vertices has a quasi-kernel \(Q\) such that
\[
|N_D^+[Q]|\ge \sqrt n .
\]
\end{theorem}

The paper is organized as follows. In Section~2, we introduce notation and
basic definitions. In Section~3, we develop the weighted half-neighborhood
framework and establish a general sinks proposition. Section~4 contains the
general \(2n/3\) theorem and its application to break digraphs. In Section~5,
we study two special subclasses of break digraphs admitting stronger
\(n/2\) bounds. Finally, Section~6 contains the
\(\sqrt n\) lower bound for large quasi-kernels.
\section{Preliminaries}

All digraphs in this paper are finite and loopless. For a digraph \(D\), we
write \(V(D)\) and \(A(D)\) for its vertex set and arc set, respectively. If
\(X\subseteq V(D)\), then \(D[X]\) denotes the subdigraph of \(D\) induced by
\(X\).

For a vertex \(v\in V(D)\), let
\[
    N_D^+(v)=\{u\in V(D): vu\in A(D)\}
    \quad\text{and}\quad
    N_D^-(v)=\{u\in V(D): uv\in A(D)\}.
\]
For a set \(X\subseteq V(D)\), let
\[
    N_D^+(X)=
    \{u\in V(D)\setminus X:\text{ there exists }x\in X\text{ with }xu\in A(D)\},
\]
and define the closed out-neighborhood of \(X\) by
\[
    N_D^+[X]=X\cup N_D^+(X).
\]
Similarly, one may define \(N_D^-(X)\) and \(N_D^-[X]\).

A vertex \(v\) is called a source if \(d_D^-(v)=0\), and a sink if
\(d_D^+(v)=0\). A digraph is source-free if it has no source.

A set \(X\subseteq V(D)\) is independent if there is no arc between any two
vertices of \(X\). A \emph{quasi-kernel} of \(D\) is an independent set
\(K\subseteq V(D)\) such that every vertex of \(D\) is reachable from some
vertex of \(K\) by a directed path of length at most 2.

For a weight function \(f:V(D)\to \mathbb Z_{\ge0}\) and a set
\(X\subseteq V(D)\), we write
\[
    f(X)=\sum_{x\in X} f(x).
\]

An \emph{out-kernel} of a digraph \(D\) is an independent set \(K\subseteq V(D)\)
such that \(V(D)=N_D^+[K]\). Equivalently, every vertex outside \(K\) has an
in-neighbor in \(K\). This is a kernel of the reverse digraph \(D^{\mathrm{rev}}\).

A vertex \(v\) of a digraph \(D\) is a \emph{king} if every vertex of \(D\) is
reachable from \(v\) by a directed path of length at most \(2\).

\begin{lemma}\label{lem:tournament-king}
Every nonempty tournament has a king.
\end{lemma}

\begin{proof}
Let \(v\) be a vertex of maximum out-degree in a tournament \(T\). If
\(u\notin N_T^+[v]\), then \(u\to v\). If every vertex of \(N_T^+(v)\) were also
an out-neighbor of \(u\), then \(u\) would have larger out-degree than \(v\), a
contradiction. Hence some \(w\in N_T^+(v)\) satisfies \(w\to u\), and so
\(v\to w\to u\). Thus \(v\) is a king.
\end{proof}

We will use the following classical theorem of Chv\'atal and Lov\'asz.

\begin{theorem}[Chv\'atal--Lov\'asz]\label{thm:CL}
Every finite digraph has a quasi-kernel.
\end{theorem}

An oriented graph is a digraph with no pair of opposite arcs. A tournament is an
orientation of a complete graph. A digraph is acyclic if it has no directed
cycle.

\begin{definition}[Break digraphs]
An oriented graph \(D\) is called a \emph{break digraph} if its vertex set can
be partitioned as
\[
    V(D)=S\dot\cup T
\]
such that \(D[S]\) is acyclic and \(D[T]\) is a tournament. We call such a
partition a \emph{break} of \(D\).
\end{definition}

The following simple observation will be used repeatedly.

\begin{lemma}\label{lem:hereditary-break}
Every induced subdigraph of a break digraph is again a break digraph.
\end{lemma}

\begin{proof}
Let \(D\) be a break digraph with break \(V(D)=S\dot\cup T\), where \(D[S]\) is
acyclic and \(D[T]\) is a tournament. For \(X\subseteq V(D)\), we have
\[
    X=(X\cap S)\dot\cup (X\cap T).
\]
The induced subdigraph \(D[X\cap S]\) is acyclic, and \(D[X\cap T]\) is a
tournament. Hence \(D[X]\) is again a break digraph.
\end{proof}

We will also use the following notion in Section~\ref{sink-version}.

\begin{definition}[Blow-up]
Let \(H\) be an oriented graph. A \emph{blow-up} of \(H\) is obtained by
replacing each vertex \(v\in V(H)\) with a nonempty independent set \(B_v\), and
by replacing each arc \(uv\in A(H)\) with all arcs from \(B_u\) to \(B_v\). Thus
\[
    V(G)=\dot\bigcup_{v\in V(H)}B_v,
\]
and for distinct \(u,v\in V(H)\), all arcs between \(B_u\) and \(B_v\) are
directed from \(B_u\) to \(B_v\) if \(uv\in A(H)\), from \(B_v\) to \(B_u\) if
\(vu\in A(H)\), and there are no arcs between \(B_u\) and \(B_v\) if neither
arc is present in \(H\).
\end{definition}

\section{Weighted large quasi-kernels and sinks}\label{sink-version}

The main idea of this paper is to study hereditary classes satisfying a weighted large quasi-kernel property. In this section, we introduce an abstract weighted half-neighborhood condition and show that it yields strong structural consequences. In particular, we establish a general proposition relating weighted large quasi-kernels to the existence of sinks, which will serve as the key ingredient in the later \(2n/3\) theorem.

\begin{definition}
A hereditary class \(\mathcal G\) of oriented graphs is said to satisfy the
\emph{weighted half-neighborhood property} if for every \(H\in\mathcal G\) and
every function \(f:V(H)\to\mathbb Z_{\ge0}\), there exists a quasi-kernel
\(Q\) of \(H\) such that
\[
    f(N_H^+[Q])\ge \frac{f(V(H))}{2}.
\]
\end{definition}

The following proposition will be the main tool in the next section.

\begin{proposition}\label{prop:sinks-general}
Let \(\mathcal G\) be a hereditary class satisfying the weighted
half-neighborhood property. Let \(D\in\mathcal G\), and let
\[
    Z=\{v\in V(D): v \text{ is a sink and not a source in }D\},
    \qquad z=|Z|.
\]
Then \(D\) has a quasi-kernel \(K\) such that
\[
    |K|\le |V(D)|-\left\lceil\frac z2\right\rceil .
\]
\end{proposition}

\begin{proof}
If \(z=0\), the statement is immediate. Assume \(z\ge1\), and put
\[
    A=V(D)\setminus Z,
    \qquad
    H=D[A].
\]
Since \(\mathcal G\) is hereditary, we have \(H\in\mathcal G\).

For each \(v\in Z\), choose an in-neighbor \(a_v\in A\). This is possible
because \(v\) is not a source, and every in-neighbor of \(v\) lies in \(A\),
as vertices of \(Z\) are sinks.

Define \(f:V(H)\to\mathbb Z_{\ge0}\) by
\[
    f(a)=|\{v\in Z:a_v=a\}|.
\]
Then
\[
    f(V(H))=z.
\]

By the weighted half-neighborhood property, \(H\) has a quasi-kernel
\(K_A\) such that
\[
    f(N_H^+[K_A])\ge \frac z2.
\]

Let
\[
    Z_{\mathrm{bad}}
    =
    \{v\in Z:N_D^-(v)\cap N_H^+[K_A]=\varnothing\}.
\]
Every vertex of \(Z_{\mathrm{bad}}\) contributes through its chosen
in-neighbor to \(f(V(H)\setminus N_H^+[K_A])\). Hence
\[
    |Z_{\mathrm{bad}}|
    \le
    f(V(H)\setminus N_H^+[K_A])
    \le
    \frac z2.
\]

Define
\[
    K=K_A\cup Z_{\mathrm{bad}}.
\]

We claim that \(K\) is a quasi-kernel of \(D\).
The set \(K\) is independent because \(K_A\) is independent,
vertices of \(Z_{\mathrm{bad}}\) are sinks,
and no vertex of \(K_A\) sends an arc to a vertex of
\(Z_{\mathrm{bad}}\).

Moreover, \(K_A\) 2-dominates \(A\).
If \(v\in Z\setminus Z_{\mathrm{bad}}\), then \(v\) has an in-neighbor in
\(N_H^+[K_A]\), so \(v\) is reached from \(K_A\) within distance at most 2.
Finally, every vertex of \(Z_{\mathrm{bad}}\) belongs to \(K\).
Thus \(K\) is a quasi-kernel of \(D\).

Finally,
\[
    |K|
    \le
    |A|+\frac z2
    =
    |V(D)|-\frac z2.
\]
Since \(|K|\) is an integer,
\[
    |K|
    \le
    |V(D)|-\left\lceil\frac z2\right\rceil .
\]
\end{proof}

The following theorem of Nguyen--Scott--Seymour implies that break digraphs
satisfy the weighted half-neighborhood property.

\begin{theorem}[Nguyen--Scott--Seymour {\cite[Theorem 3.2]{NSS}}]\label{thm:NSS}
Let \(G\) be an oriented graph that admits a break, and let
\(f:V(G)\to\mathbb Z_{\ge0}\).
Then \(G\) has a quasi-kernel \(K\) such that
\[
    f(N_G^+[K])\ge \frac{f(V(G))}{2}.
\]
\end{theorem}

\begin{remark}
In \cite[Theorem 3.2]{NSS}, the function \(f\) is assumed to be positive.
The same statement for nonnegative functions follows by applying the theorem
to \(Nf+1\) with \(N\) sufficiently large.
\end{remark}

Since the class of break digraphs is hereditary
(Lemma~\ref{lem:hereditary-break}),
Theorem~\ref{thm:NSS} implies that break digraphs satisfy the weighted
half-neighborhood property. Consequently,
Proposition~\ref{prop:sinks-general} applies to every break digraph.
\begin{corollary}
The class of break digraphs satisfies the weighted half-neighborhood property.
\end{corollary}

The weighted half-neighborhood property is also stable under natural graph operations. We next show that it is preserved under blow-ups, yielding many additional hereditary examples satisfying the same weighted conclusion.

\begin{theorem}[Weighted half-neighborhood property is preserved by blow-up]\label{thm:blowup}
Let $H$ be an oriented graph. Assume that $H$ satisfies the following property:
for every weight function $F:V(H)\to \mathbb{Z}_{\ge0}$, there exists a quasi-kernel $K_H$ of $H$ such that
\[
F\big(N_H^{+}[K_H]\big)\ \ge\ \frac{F(V(H))}{2}.
\]
Then every blow-up $G$ of $H$ satisfies the analogous weighted property:
for every weight function $f:V(G)\to \mathbb{Z}_{\ge0}$, there exists a quasi-kernel $K$ of $G$ such that
\[
f\big(N_G^{+}[K]\big)\ \ge\ \frac{f(V(G))}{2}.
\]
In particular, if $H$ admits a break, then every blow-up of $H$ satisfies the weighted
conclusion of Nguyen--Scott--Seymour~\cite[Theorem~3.2]{NSS}.
\end{theorem}

\begin{proof}
Let $G$ be a blow-up of $H$ with parts $\{B_v: v\in V(H)\}$.
Fix any weight function $f:V(G)\to \mathbb{Z}_{\ge0}$.
Define a weight function $F:V(H)\to \mathbb{Z}_{\ge0}$ by
\[
F(v)\ :=\ \sum_{x\in B_v} f(x)\qquad (v\in V(H)).
\]
Clearly,
\[
F(V(H))=\sum_{v\in V(H)}F(v)=\sum_{x\in V(G)}f(x)=f(V(G)).
\]

By the hypothesis on $H$, there exists a quasi-kernel $K_H\subseteq V(H)$ such that
\[
F\big(N_H^{+}[K_H]\big)\ \ge\ \frac{F(V(H))}{2}.
\]
Now define
\[
K \ :=\ \bigcup_{v\in K_H} B_v \ \subseteq V(G).
\]

\smallskip
\noindent\textbf{Claim 1: $K$ is independent in $G$.}
Since \(K_H\) is independent in \(H\), there is no arc between \(B_u\) and \(B_v\) whenever \(u,v\in K_H\) are distinct. Each \(B_u\) is independent by definition of blow-up. Hence \(K\) is independent.

\smallskip
\noindent\textbf{Claim 2: $K$ is a quasi-kernel of $G$.}
Let $y\in V(G)$ and choose $w\in V(H)$ with $y\in B_w$.
Since $K_H$ is a quasi-kernel of $H$, there exists $v\in K_H$ such that $w$ is reachable from $v$
by a directed path in $H$ of length at most $2$.
If $v=w$, then $y\in B_v\subseteq K$.
If $v\to w$ is an arc of $H$, then every vertex of $B_v\subseteq K$ has an arc to every vertex of $B_w$,
so $y$ is reached from $K$ in one step.
If $v\to x\to w$ is a directed path of length $2$ in $H$, then by the blow-up definition we have
all arcs $B_v\to B_x$ and all arcs $B_x\to B_w$ in $G$, so $y$ is reached from $K$ in two steps.
Hence every vertex of $G$ is reachable from $K$ by a directed path of length at most $2$, and $K$ is a quasi-kernel.

\smallskip
\noindent\textbf{Claim 3:
\[
N_G^{+}[K]\ =\ \bigcup_{u\in N_H^{+}[K_H]} B_u.
\]}
Indeed, a vertex of $G$ lies in $N_G^{+}[K]$ precisely if it is in some $B_u$ where either
$u\in K_H$ (so $B_u\subseteq K$) or there exists $v\in K_H$ with an arc $v\to u$ in $H$,
in which case all arcs $B_v\to B_u$ appear in $G$ and thus $B_u\subseteq N_G^{+}(K)$.
This gives the claimed equality, and therefore
\[
f\big(N_G^{+}[K]\big)
=\sum_{u\in N_H^{+}[K_H]}\ \sum_{x\in B_u} f(x)
=\sum_{u\in N_H^{+}[K_H]} F(u)
=F\big(N_H^{+}[K_H]\big).
\]

Combining the above identities yields
\[
f\big(N_G^{+}[K]\big)
=F\big(N_H^{+}[K_H]\big)
\ \ge\ \frac{F(V(H))}{2}
=\frac{f(V(G))}{2}.
\]
This completes the proof.
\end{proof}

\section{A \texorpdfstring{$2n/3$}{2n/3} bound in hereditary classes}

We now combine Proposition~\ref{prop:sinks-general} with a matching argument to obtain the main structural result of this paper. The theorem below shows that every source-free digraph in a hereditary class satisfying the weighted half-neighborhood property admits a quasi-kernel of size at most \(2n/3\).

\begin{theorem}\label{thm:general-2n3}
Let \(\mathcal G\) be a hereditary class satisfying the weighted
half-neighborhood property.
Let \(D\in\mathcal G\) be source-free, and let \(n=|V(D)|\).
Then \(D\) has a quasi-kernel \(K\) such that
\[
    |K|\le \frac{2n}{3}.
\]
\end{theorem}

\begin{proof}
By Proposition~\ref{prop:sinks-general}, every induced subdigraph
\(H\in\mathcal G\) with \(z\) sinks that are not sources
has a quasi-kernel \(Q\) satisfying
\[
    |Q|
    \le
    |V(H)|-\left\lceil\frac z2\right\rceil .
\]

Choose an inclusion-minimal quasi-kernel \(K\subseteq V(D)\). Let
\[
    N=N_D^+(K),
    \qquad
    M=V(D)\setminus (K\cup N).
\]
Thus
\[
    V(D)=K\dot\cup N\dot\cup M.
\]

Consider the bipartite graph with parts \(K\) and \(N\), where
\(ku\) is an edge whenever \(k\to u\) in \(D\).
Let \(\mathcal M\) be a maximum matching.
Let \(K_1\subseteq K\) be the set of vertices saturated by \(\mathcal M\),
and let
\[
    K_2=K\setminus K_1.
\]

Write
\[
    r=|K_1|,
    \qquad
    s=|K_2|,
    \qquad
    p=|N|,
    \qquad
    m=|M|.
\]
Then
\[
    n=r+s+p+m,
\]
and since \(\mathcal M\) is a matching,
\[
    r\le p.
\]

\medskip
\noindent\textbf{Claim 1.}
Every vertex \(u\in N\) has an in-neighbor in \(K_1\).

\begin{proof}
By definition of \(N=N_D^+(K)\), every \(u\in N\) has an in-neighbor in \(K\).
If none of these in-neighbors belonged to \(K_1\), then one of them would lie
in \(K_2\), and matching this vertex to \(u\) would enlarge \(\mathcal M\),
a contradiction.
\end{proof}

\medskip
\noindent\textbf{Claim 2.}
No vertex of \(K_2\) has an in-neighbor in \(N\).

\begin{proof}
Suppose, to the contrary, that \(u\to v\) for some \(u\in N\) and
\(v\in K_2\). By Claim~1, there exists \(w\in K_1\) such that \(w\to u\).
Then
\[
    w\to u\to v,
\]
so \(v\) is reached from \(K\setminus\{v\}\) within distance at most 2.

We claim that \(K\setminus\{v\}\) is still a quasi-kernel of \(D\),
contradicting the inclusion-minimality of \(K\).
The set \(K\setminus\{v\}\) remains independent.
Moreover, every vertex of \(N_D^+(v)\) lies in \(N\), and every vertex of \(N\)
has an in-neighbor in \(K_1\) by Claim~1. Hence all former out-neighbors of
\(v\) are still reached from \(K\setminus\{v\}\) within distance at most \(1\),
while \(v\) itself is reached by the directed path \(w\to u\to v\).
Thus \(K\setminus\{v\}\) is a quasi-kernel of \(D\), a contradiction.
\end{proof}

Now consider
\[
    H=D[K_2\cup M].
\]
Since \(\mathcal G\) is hereditary and \(D\in\mathcal G\),
we have \(H\in\mathcal G\).

Every vertex of \(K_2\) is a sink in \(H\), because \(K\) is independent and
all out-neighbors of vertices of \(K\) lie in \(N\).
Moreover, no vertex of \(K_2\) is a source in \(H\).
Indeed, since \(D\) is source-free, every vertex of \(K_2\) has an
in-neighbor in \(D\). This in-neighbor cannot lie in \(K\), by the
independence of \(K\), and cannot lie in \(N\), by Claim~2.
Hence it lies in \(M\).

Therefore \(H\) has at least \(s=|K_2|\) vertices that are sinks but not
sources. By Proposition~\ref{prop:sinks-general}, \(H\) has a quasi-kernel
\(K'\) such that
\[
    |K'|
    \le
    |V(H)|-\left\lceil\frac s2\right\rceil
    \le
    m+\frac s2.
\]

Define
\[
    K^*
    =
    K'
    \cup
    \bigl(K_1\setminus N_D^+(K')\bigr).
\]

We claim that \(K^*\) is a quasi-kernel of \(D\).

First, \(K^*\) is independent.
The set \(K'\) is independent, and
\(K_1\setminus N_D^+(K')\) is independent because it is contained in \(K\).
There is no arc from \(K'\) to
\(K_1\setminus N_D^+(K')\) by definition.
There is also no arc from \(K_1\setminus N_D^+(K')\) to \(K'\).
Indeed, if \(x\in K_1\) sent an arc to some \(y\in K'\), then
\(y\notin K_2\), since \(K\) is independent, and
\(y\notin M\), since then \(y\in N_D^+(K)\), contradicting \(y\in M\).

Next, \(K^*\) 2-dominates \(D\).
Since \(K'\) is a quasi-kernel of \(H=D[K_2\cup M]\),
every vertex of \(K_2\cup M\) is reached from \(K'\) within distance at most 2.

Every vertex of \(K_1\) is reached from \(K^*\) within distance at most \(1\):
either it belongs to \(K_1\setminus N_D^+(K')\), or it lies in \(N_D^+(K')\).

Finally, if \(u\in N\), then by Claim~1 there exists \(w\in K_1\)
with \(w\to u\).
Since \(w\) is reached from \(K^*\) within distance at most \(1\),
the vertex \(u\) is reached from \(K^*\) within distance at most 2.
Thus \(K^*\) is a quasi-kernel of \(D\).

Moreover,
\[
    |K^*|
    \le
    |K'|+|K_1|
    \le
    m+\frac s2+r.
\]

It remains to compare \(K\) and \(K^*\).
Suppose, for a contradiction, that both have size greater than \(2n/3\).

Since \(|K|=r+s\), we have
\[
    r+s>\frac{2(r+s+p+m)}{3}.
\]
Thus
\[
    r+s>2p+2m.
\]
Using \(p\ge r\), this gives
\[
    s>2m.
\]

On the other hand, since
\[
    |K^*|\le r+m+\frac s2
\]
and \(|K^*|>2n/3\), we obtain
\[
    r+m+\frac s2>\frac{2(r+s+p+m)}{3}.
\]
Equivalently,
\[
    2r+2m-s>4p.
\]
Using \(p\ge r\), this implies
\[
    2r+2m-s>4r,
\]
and hence
\[
    s<2m-2r\le 2m.
\]
This contradicts \(s>2m\).

Therefore at least one of \(K\) and \(K^*\) has size at most \(2n/3\).
Since both are quasi-kernels of \(D\), the result follows.
\end{proof}

\begin{corollary}
Every source-free break digraph \(D\) on \(n\) vertices has a quasi-kernel
\(K\) such that
\[
    |K|\le \frac{2n}{3}.
\]
\end{corollary}

\begin{proof}
The class of break digraphs is hereditary by
Lemma~\ref{lem:hereditary-break}, and satisfies the weighted
half-neighborhood property by Theorem~\ref{thm:NSS}.
Hence the result follows from Theorem~\ref{thm:general-2n3}.
\end{proof}

\section{Two special cases of break digraphs}

The general \(2n/3\) theorem naturally raises the question whether stronger
bounds hold for more structured subclasses of break digraphs. In this section
we record two situations in which one can obtain an \(n/2\) bound.

\begin{definition}
A digraph \(D\) is a \emph{one-way break digraph} if \(V(D)\) has a partition
\(V(D)=S\dot\cup T\) such that \(D[S]\) is acyclic, \(D[T]\) is a tournament,
and every arc between \(S\) and \(T\) is directed from \(T\) to \(S\).
\end{definition}

\begin{theorem}\label{thm:one-way-break}
Let \(D=S\dot\cup T\) be a source-free one-way break digraph on \(n\) vertices,
where \(D[S]\) is acyclic and \(D[T]\) is a tournament. Then \(D\) has a
quasi-kernel \(K\) with \(|K|\le n/2\).
\end{theorem}

\begin{proof}
We argue by induction on \(|A(D[S])|\). If \(|A(D[S])|=0\), then \(D\) is a
one-way split digraph, and the assertion follows from the result of
Ai, Gerke, Gutin, Yeo and Zhou~\cite{AGGYZ}.

Assume now that \(|A(D[S])|\ge1\). Let \(v_0v_1\ldots v_k\) be a longest
directed path in \(D[S]\), and put
\[
    X=N_D^+(v_{k-1}).
\]
Since all arcs between \(S\) and \(T\) are directed from \(T\) to \(S\), every
out-neighbor of \(v_{k-1}\) lies in \(S\). Moreover, every vertex of \(X\) is a
sink in \(D[S]\). Indeed, if \(x\in X\) had an out-neighbor \(y\) in \(D[S]\),
then either \(v_0v_1\ldots v_{k-1}xy\) would extend the chosen path, or \(y\)
would lie earlier on the path and create a directed cycle in \(D[S]\). Both are
impossible. In particular, \(X\ne\varnothing\), since \(v_k\in X\).

Let
\[
    D'=D-\bigl(\{v_{k-1}\}\cup X\bigr).
\]
Then \(D'\) is again a source-free one-way break digraph. To see this, suppose
that some \(u\in V(D')\) became a source in \(D'\). Then every in-neighbor of
\(u\) deleted from \(D\) lies in \(\{v_{k-1}\}\cup X\). But \(v_{k-1}\) sends
arcs only to vertices of \(X\), and vertices of \(X\) are sinks in \(D[S]\) and
have no out-neighbors in \(T\). Hence no deleted vertex sends an arc to \(u\), a
contradiction. Also,
\[
    |A(D'[S])|<|A(D[S])|,
\]
because the arc \(v_{k-1}v_k\) is deleted. By induction, \(D'\) has a
quasi-kernel \(K'\) such that
\[
    |K'|\le \frac{|V(D')|}{2}=\frac{n-1-|X|}{2}.
\]

If \(v_{k-1}\) has an in-neighbor \(u\in K'\), then \(K'\) is already a
quasi-kernel of \(D\): the vertex \(v_{k-1}\) is reached by \(u\to v_{k-1}\),
and every \(x\in X\) is reached by \(u\to v_{k-1}\to x\).

Otherwise, set
\[
    K=K'\cup\{v_{k-1}\}.
\]
This set is independent, since \(N_D^+(v_{k-1})=X\) is disjoint from \(K'\) and
\(N_D^-(v_{k-1})\cap K'=\varnothing\) by assumption. It is also a quasi-kernel:
vertices of \(D'\) are still reached from \(K'\), the vertex \(v_{k-1}\) belongs
to \(K\), and every vertex of \(X\) is reached from \(v_{k-1}\) in one step.
Finally,
\[
    |K|\le \frac{n-1-|X|}{2}+1\le \frac n2,
\]
because \(|X|\ge1\). This completes the induction.
\end{proof}

\begin{theorem}\label{thm:outkernel-side}
Let \(D=S\dot\cup T\) be a source-free break digraph on \(n\) vertices, where
\(D[S]\) is acyclic and \(D[T]\) is a tournament. Suppose that \(D[S]\) has an
out-kernel \(K\) such that
\[
    |K|+1\le \frac n2.
\]
Then \(D\) has a quasi-kernel \(Q\) with \(|Q|\le n/2\). If in addition
\(T\subseteq N_D^+[K]\), the weaker assumption \(|K|\le n/2\) suffices.
\end{theorem}

\begin{proof}
If \(T\subseteq N_D^+[K]\), then \(K\) itself is a quasi-kernel of \(D\): it
reaches every vertex of \(S\) in at most one step because \(K\) is an out-kernel
of \(D[S]\), and it reaches every vertex of \(T\) in at most one step by the
assumption \(T\subseteq N_D^+[K]\). Thus the final assertion follows.

We may therefore assume that
\[
    V_0=T\setminus N_D^+[K]
\]
is nonempty. Since \(D[V_0]\) is a tournament, Lemma~\ref{lem:tournament-king}
gives a king \(v_0\in V_0\). Define
\[
    Q=(K\setminus N_D^+(v_0))\cup\{v_0\}.
\]
The set \(Q\) is independent. Indeed, \(K\setminus N_D^+(v_0)\) is independent;
there is no arc from \(v_0\) to \(K\setminus N_D^+(v_0)\) by definition; and
there is no arc from \(K\setminus N_D^+(v_0)\) to \(v_0\), because
\(v_0\notin N_D^+[K]\).

We next verify that \(Q\) 2-dominates \(D\). Let \(s\in S\). Since \(K\) is an
out-kernel of \(D[S]\), either \(s\in K\) or there is some \(k\in K\) with
\(k\to s\). If such a vertex \(k\) remains in \(Q\), then \(s\) is reached from
\(Q\) in at most one step. If \(k\notin Q\), then \(v_0\to k\), and hence
\(v_0\to k\to s\) reaches \(s\) in at most two steps. Thus every vertex of
\(S\) is reached from \(Q\) within distance at most \(2\).

Now let \(t\in T\setminus V_0\). Then \(t\in N_D^+[K]\), so some \(k\in K\)
satisfies \(k\to t\). As above, either \(k\in Q\), giving a one-step path to
\(t\), or \(v_0\to k\to t\), giving a two-step path. Finally, every vertex of
\(V_0\) is reached from \(v_0\) within distance at most \(2\), since \(v_0\) is a
king of \(D[V_0]\). Hence \(Q\) is a quasi-kernel of \(D\).

The size bound follows from
\[
    |Q|\le |K|+1\le \frac n2.
\]
\end{proof}

\section{A \texorpdfstring{$\sqrt n$}{sqrt(n)} bound for the Large Quasi-kernel Conjecture}
In this section, we consider a different aspect of large quasi-kernels.
Instead of bounding the size of a quasi-kernel from above, we study how large
its closed out-neighborhood can be. This direction is motivated by the
Large Quasi-kernel Conjecture, studied by Spiro~\cite{Spiro}, which asserts
that every digraph \(D\) has a quasi-kernel \(Q\) satisfying
\[
|N_D^+[Q]|\ge \frac{|V(D)|}{2}.
\]
Previously, the best known general lower bound for quasi-kernels was
\[
|N_D^+[Q]|\ge |V(D)|^{1/3},
\]
proved by Spiro~\cite[Theorem 2.8]{Spiro}. We prove a general
theorem for hereditary classes, which improves this to a \(\sqrt n\) lower
bound for arbitrary digraphs.

\begin{theorem}\label{thm:amplification}
Let \(\mathcal G\) be a hereditary class of digraphs, and let \(c\in[0,1]\).
Suppose that every digraph \(D\in\mathcal G\) on \(n\) vertices has a
quasi-kernel \(Q\) with
\[
    |Q|\ge n^c .
\]
Then every digraph \(D\in\mathcal G\) on \(n\) vertices has a quasi-kernel \(Q\)
such that
\[
    |N_D^+[Q]|\ge n^{(1+c)/2}.
\]
\end{theorem}

\begin{proof}
Set
\[
    \alpha=\frac{1+c}{2}.
\]
We prove the result by induction on \(n=|V(D)|\). The case \(n=1\) is trivial.
Assume the result holds for all digraphs in \(\mathcal G\) with fewer than
\(k\) vertices, and let \(D\in\mathcal G\) have \(k\) vertices.

By assumption, \(D\) has a quasi-kernel \(Q\) with
\[
    |Q|\ge k^c.
\]
Set
\[
    Q_1=N_D^+(Q), \qquad
    Q_2=V(D)\setminus (Q\cup Q_1),
\]
and write
\[
    q=|Q|,\qquad q_1=|Q_1|,\qquad q_2=|Q_2|.
\]
If
\[
    |N_D^+[Q]|=q+q_1\ge k^\alpha,
\]
then \(Q\) is the desired quasi-kernel. Hence we may assume
\[
    q+q_1<k^\alpha.
\]
Thus
\[
    q_2=k-q-q_1>k-k^\alpha.
\]

Since \(\mathcal G\) is hereditary, \(D[Q_2]\in\mathcal G\). By the induction
hypothesis, \(D[Q_2]\) has a quasi-kernel \(Q_0\) such that
\[
    |N_{D[Q_2]}^+[Q_0]|\ge q_2^\alpha.
\]

Define
\[
    Q'=Q_0\cup \bigl(Q\setminus N_D^+(Q_0)\bigr).
\]
We claim that \(Q'\) is a quasi-kernel of \(D\).

First, \(Q'\) is independent. Indeed, \(Q_0\) is independent in \(D[Q_2]\), and
\(Q\setminus N_D^+(Q_0)\) is independent because \(Q\) is independent. Moreover,
there is no arc from \(Q_0\) to \(Q\setminus N_D^+(Q_0)\) by construction, and
there is no arc from \(Q\) to \(Q_0\), since \(Q_0\subseteq Q_2\) and
\(Q_2\cap N_D^+(Q)=\varnothing\).

Now we check the 2-domination property. Every vertex of \(Q_2\) is reachable
from \(Q_0\) within directed distance at most 2. Every vertex of \(Q\) is
reachable from \(Q'\) within distance at most \(1\): if a vertex of \(Q\) is not
in \(Q'\), then it lies in \(N_D^+(Q_0)\). Finally, every vertex of
\(Q_1=N_D^+(Q)\) is reachable from \(Q'\) within distance at most 2, because
every vertex of \(Q\) is reachable from \(Q'\) within distance at most \(1\).
Hence \(Q'\) is a quasi-kernel of \(D\).

It remains to estimate its closed out-neighborhood. Since
\(N_{D[Q_2]}^+[Q_0]\subseteq Q_2\), \(Q\cap Q_2=\varnothing\), and
\(Q\subseteq N_D^+[Q']\), we have
\[
    |N_D^+[Q']|
    \ge |N_{D[Q_2]}^+[Q_0]|+|Q|
    \ge q_2^\alpha+q.
\]
Using \(q_2>k-k^\alpha\) and \(q\ge k^c\), we obtain
\[
    |N_D^+[Q']|
    > (k-k^\alpha)^\alpha+k^c.
\]
Since \(c=2\alpha-1\), we have
\[
    |N_D^+[Q']|>(k-k^\alpha)^\alpha+k^{2\alpha-1} = k^\alpha \left((1-k^{\alpha - 1})^\alpha + k ^{\alpha -1}\right).
\]
Moreover, as \(0\le \alpha\le 1\),
\[
    (1-x)^\alpha\ge 1-x
\]
for \(0\le x\le 1\). Taking \(x=k^{\alpha-1}\), we get
\[
    |N_D^+[Q']|> k^\alpha(1-k ^{\alpha -1}+k ^{\alpha -1}) = k^\alpha.
\]
This completes the induction.
\end{proof}
\begin{corollary}
Let \(D\) be a digraph on \(n\) vertices. Then \(D\) has a quasi-kernel \(Q\)
such that
\[
    |N_D^+[Q]|\ge \sqrt n .
\]
\end{corollary}

\begin{proof}
Apply Theorem~\ref{thm:amplification} to the hereditary class of all digraphs
with \(c=0\). By Theorem~\ref{thm:CL}, every digraph has a quasi-kernel, and hence has one of
size at least \(1=n^0\). Therefore there exists a quasi-kernel \(Q\) with
\[
    |N_D^+[Q]|\ge n^{(1+0)/2}=\sqrt n .
\]
\end{proof}

\section{Concluding remarks}

In this paper, we developed a weighted hereditary framework for studying quasi-kernels. Using a sinks reduction and a matching argument, we proved a \(2n/3\) bound for source-free digraphs in hereditary classes satisfying the weighted half-neighborhood property, and applied this to break digraphs. We also obtained \(n/2\) bounds for two special subclasses of break digraphs, and improved the general large quasi-kernel lower bound from \(n^{1/3}\) to \(\sqrt n\).

We conclude with several open questions.
A theorem of Chv\'atal and Lov\'asz states that every digraph has a quasi-kernel,
while the Small Quasi-kernel Conjecture predicts the existence of a quasi-kernel
of size at most \(n/2\). Our results suggest studying this conjecture for
digraphs admitting a partition
\[
V(D)=S\dot\cup T,
\]
where \(D[T]\) is a tournament and \(D[S]\) has a small out-kernel. The results of
Section~5 show that this case can be separated from the remaining difficulty.

Another natural hereditary class is given by digraphs whose underlying graph is
triangle-free. By Ramsey theory, every triangle-free graph has an independent
set of size at least \(\sqrt n\). It would be interesting to investigate both
the small and large quasi-kernel conjectures for this class.

Finally, our results provide evidence for the following conjecture.

\begin{conjecture}
Every break digraph on \(n\) vertices has a quasi-kernel of size at most
\(\frac n2\).
\end{conjecture}

\end{document}